\documentclass[11pt]{amsart}
\usepackage{amsmath,amssymb}

\usepackage[hyperindex]{hyperref}

\usepackage{amsrefs,booktabs,verbatim}
\usepackage{xcolor}

\title{An introduction to upper half plane polynomials}
\author{Steve Fisk}
\date{\today}

\newtheorem{lemma}{Lemma}
\newtheorem{fact}{Fact}
\newtheorem{theorem}{Theorem}
\newtheorem{question}{Question}
\theoremstyle{definition}
\newtheorem{definition}{Definition}
\newtheorem{remark}{Remark}
\newtheorem{example}{Example}

\newenvironment{aside}{}{}

\newcommand{\Mobius}{M\"{o}bius}
\newcommand{\aaa}{\mathbf{a}}
\newcommand{\bbb}{\mathbf{b}}
\newcommand{\complexes}{{\mathbb C}}
\newcommand{\diffi}{{\text{{\textbf{{I}}}}}} 
\newcommand{\diffj}{{\text{{\textbf{{J}}}}}} 
\newcommand{\imag}{\boldsymbol{\imath}}
\newcommand{\ulace}{\ensuremath{\stackrel{U}{\longleftarrow}}}
\newcommand{\reals}{{\mathbb R}}
\newcommand{\sdiffi}{{\text{{\textbf{\tiny{I}}}}}} 
\newcommand{\sdiffj}{{\text{{\textbf{\tiny{J}}}}}} 
\newcommand{\sdiffk}{{\text{{\textbf{\tiny{K}}}}}} 
\newcommand{\upbar}[1]{\widehat{\text{\texttt{UP}}}_{#1}}
\newcommand{\up}[1]{\text{\texttt{U}}_{#1}\,(\complexes)}
\newcommand{\rup}[1]{{\text{\texttt{U}}_{#1}}}
\newcommand{\xx}{\mathbf{x}}
\newcommand{\yy}{\mathbf{y}}
\newcommand{\uu}{\mathbf{u}}
\newcommand{\vv}{\mathbf{v}}

\newcommand{\smalltwodet}[4]{%
\left|\begin{smallmatrix} #1&#2 \\ #3&#4\end{smallmatrix}\right|}

\begin{document}
\maketitle

\begin{aside}
Polynomials with all real roots have many interesting and useful
properties. The purpose of this article is to introduce a
generalization to polynomials in many variables and with complex coefficients
\cites{bbs,bb2,bbl,bbs-johnson,branden-hpp,fisk,wagner}. 
\end{aside}

\begin{definition}
  $\up{d} = \left\{\text{\parbox{3.5in}{All polynomials
      $f(x_1,\dots,x_d)$ with complex coefficients such that
      $f(\sigma_1,\dots,\sigma_d)\ne0$ for all
      $\sigma_1,\dots,\sigma_d$ in the upper half plane. If we don't
      need to specify $d$ we simply write $\up{}$.}}\right.$ We call
such polynomials upper half plane polynomials, or simply \emph{upper polynomials}.

\end{definition}

\begin{aside}
For example, $x_1+\cdots+ x_d\in\up{d}$. This follows from the fact
that the upper half plane is a cone, so if $\sigma_1,\dots,\sigma_d$
are in the upper half plane then so is their sum.

Another example is $x_1x_2-1$. If $\sigma_1$ and $\sigma_2$ are in the
upper half plane then
$\sigma_1\sigma_2\in\complexes\setminus(0,\infty)$, so
$\sigma_1\sigma_2-1$ is not zero.

$\up{1}$ is easily described. It is all polynomials in one variable
whose roots are either real, or lie in the lower half plane.

It is important to observe that the zero polynomial is not in
$\up{}$. This is unfortunate, since it causes many conclusions to be
of the form ``\dots $\in\up{}\cup\{0\}$\dots''.
\end{aside}

\textbf{Conventions:} $d$ is a positive integer, $\imag=\sqrt{-1}$, and
$y$ is a variable distinct from $x_1,\dots,x_d$. We use the following notation

\begin{align*}
  \xx &= (x_1,\dots,x_d) &
  \yy &= (y_1,\dots,y_d)\\
  \partial_\xx&= (\partial_{x_1},\cdots,\partial_{x_d}) &
  f(\xx) &= f(x_1,\dots,x_d) \\
  \diffi &= (i_1,\dots,i_d) & \xx^\sdiffi &= (x_1^{i_1},\dots,x_d^{i_d})\\
  \diffj &= (j_1,\dots,j_d) & \xx^\sdiffj &= (x_1^{j_1},\dots,x_d^{j_d})
\end{align*}

\section{Complex Coefficients}
\label{sec:complex-coeff}

\begin{fact}\label{elem}
  Suppose $f(\xx)\in\up{d}$.
  \begin{enumerate}
  \item If $\alpha\ne0$ then $\alpha f(\xx)\in\up{d}$.  
  \item If $a_1>0,\dots,a_d>0$ then $f(a_1x_1,\dots,a_dx_d)\in\up{d}$.
    \item If $\Im(\sigma_1)>0,\dots,\Im(\sigma_d)>0$ then
      $f(x_1+\sigma_1,\dots,x_d+\sigma_d)\in\up{d}$.
    \item If $\Im(\sigma)>0$ then $f(\sigma,x_2,\dots,x_d)\in\up{}$.
    \item $f(x_1+y,x_2,\dots,x_d)\in\up{}$.
    \item $f(x,x,x_3,\dots,x_d)\in\up{}$.
  \end{enumerate}
\end{fact}
\begin{proof}
These are all immediate from the definition.  
\end{proof}

\begin{fact}\label{zero}
  If $f(\xx)\in\up{}$ then $f(a,x_2,\dots,x_d)\in\up{}\cup\{0\}$ if $a\in\reals$.
\end{fact}
\begin{proof}
  It suffices to assume $a=0$. 
  Let $g_r(\xx) = f(x_1/r,x_2,\dots,x_d)$. Since $\lim_{r\rightarrow
    \infty} g_r(\xx) = f(0,x_2,\dots,x_d)$   the Hurwitz theorem (see
  below) implies the conclusion, where $\Omega$ is the upper half plane.
\end{proof}

\begin{theorem}[Hurwitz]
  Let $(f_n)$ be a sequence of functions which are all analytic and
  without zeros in a region $\Omega$. Suppose, in addition, that
  $f_n(z)$ tends to $f(z)$, uniformly on every compact subset of
  $\Omega$. Then $f(z)$ is either identically zero or never equal to
  zero in $\Omega$.
\end{theorem}

\begin{fact}\label{multiplicatiion}
  $\up{}$ is closed under multiplication and extracting factors. 
\begin{quote}
  That is,
  $f(\xx)g(\xx)\in\up{}$ iff $f(\xx)\in\up{}$ and $g(\xx)\in\up{}$.
\end{quote}
\end{fact}
\begin{proof}
  This is also immediate from the definition of $\up{}$.
\end{proof}

\begin{aside}
Another construction that preserves upper polynomials is the reversal of one
variable. Note the introduction of a minus sign.
\end{aside}

\begin{fact}\label{reversal-1}
  If $\sum_0^n f_i(\xx)\, y^i \in\up{}$ then $\sum_0^n f_i(\xx)\,(-y)^{n-i}\in\up{}$.
\end{fact}
\begin{proof}
If $g(\xx,y) = \sum f_i(\xx) y^i$ then
\[ y^n g(\xx,-1/y) = \sum_0^n f_i(\xx) (-y)^{n-i} \]
If $\Im(\sigma)>0$ then $\Im(-\frac{1}{\sigma})>0$, so $g(\xx,-1/y)$
doesn't vanish on the upper half plane.
\end{proof}

\begin{example}
  We can  reverse some polynomials determined by matrices. If
  $X=diag(x_1,\dots,x_n)$ and $A$ is $n$ by $n$ then the reverse with
  respect to $x_1,\dots,x_n$ of $|X + A|$ is  $|-I + XA|$.
\end{example}

\begin{fact}\label{diff}
  $\up{}$ is closed under differentiation.
\begin{quote}  That is, if $f(\xx)\in\up{}$ then
  $\partial_{x_i}f(\xx)\in\up{}\cup\{0\}$. 
\end{quote}
\end{fact}
\begin{proof}
We will show that if $f(\xx)\in\up{}$ then
$\partial_{x_1}f(\xx)\in\up{}\cup\{0\}$. If
$\sigma_2,\dots,\sigma_d$ are in the upper half plane then it
    suffices to show $\frac{d}{dx}g(x)\in\up{1}$ where 
\[ g(x) = f(x,\sigma_2,\dots,\sigma_d)\]
By hypothesis $g$ has no roots in the upper half plane, so by the
Gauss-Lucas theorem all roots of $g'$ lie in the convex hull of the
roots of $g$, and so do not lie in the upper half plane.  
\end{proof}

\begin{fact}\label{coef in up}
If $\sum f_i(\xx) y^i \in\up{}$ then all coefficients $f_i(\xx)$
are in $\up{}\cup\{0\}$.
\end{fact}
\begin{proof}
  If we differentiate $i$ times with respect to $y$ we get
\[
i! f_i(\xx) + y(i+1)!f_{i+1}(\xx) + \cdots + \frac{n!}{(n-i)!}
y^{n-i}f_n(\xx)\in\up{}\cup\{0\}
\]
Now substitute $y=0$.
\end{proof}

\begin{remark}
  Coefficients can certainly be zero. Consider $x^2-1$.
\end{remark}

\begin{definition}
  We say that $f(\xx),g(\xx)$ \emph{interlace}, written $f\ulace g$, if
  and only if $f(\xx) + y g(\xx)\in\up{}$.
\end{definition}

\begin{aside}
See Remark~\ref{rem:interlacing} for the connection with the usual
definition of interlacing in terms of roots.
\end{aside}

\begin{remark}
  Positive constants interlace linear functions. That is, if
  $\ell=\sum a_i x_i+b$ where all $a_i$ are positive and $c>0$ then   $\sum a_i
  x_i+b+cy\in\up{d}$, so $\ell\ulace c$.
\end{remark}

\begin{fact}\label{adj interlace}
  Consecutive coefficients interlace.
\begin{quote} That is, if $\sum_0^n
  f_i(\xx)y^i\in\up{}$ then $f_i(\xx) \ulace f_{i+1}(\xx)$ for 
  $i=0,\dots,n-1$, provided $f_i$ and $f_{i+1}$ are not both zero.
\end{quote}
\end{fact}
\begin{proof}
Differentiate $i$ times with respect to $y$
\[
i! f_i(\xx) + y(i+1)!f_{i+1}(\xx) + \cdots + \frac{n!}{(n-i)!}
y^{n-i}f_n(\xx)\in\up{}\cup\{0\}
\]
Reverse with respect to $y$
\[
\frac{n!}{(n-i)!} f_n + \cdots + (-y)^{n-i-1}(i+1)! f_{i+1} + (-y)^{n-i}
i! f_i \in\up{}\cup\{0\}
\]
Differentiate $n-i-2$ times with respect to $y$
  \[
(n-i-1)!(i+1)! (-1)^{n-i-1} f_{i+1} + (n-i)!i! (-1)^{n-i} y f_i
\in\up{}\cup\{0\}
\]
Reversing again, factoring out constants, and rescaling $y$ yields the result.
\end{proof}

\begin{fact}\label{diff inter}
If $f\in\up{}$ then $f\ulace \partial_{x_i}f(\xx)$.   
\end{fact}
\begin{proof}
  Expanding into a Taylor series
\[ f(x_1,\dots,x_i+y,\dots,x_d) = f(\xx) + \frac{\partial f(\xx)}{\partial
  x_i} \,y + \cdots \]
shows that $f\ulace \frac{\partial f}{\partial x_i}$ since
they are consecutive coefficients of a polynomial in $\up{}$.
\end{proof}

\begin{fact}\label{basic int}\ 
  \begin{enumerate}
  \item $fg\ulace fh$ iff $f\in\up{}$ and $g\ulace h$.
  \item If $f\ulace g$ then $g \ulace -f$.
  \item If $f\ulace g$ and $f\ulace h$ then
    $f\ulace g+h$.
  \item If $f\ulace g$ and $h\ulace g$ then
    $f+h\ulace g$.
  \item If $f\ulace g\ulace h$ then $f-h\ulace g$.
  \end{enumerate}
\end{fact}
\begin{proof}
  If $fg\ulace fh$ then $fg+y fh = f(g+yh)\in\up{}$ so
  $g+yh\in\up{}$. If $f+yg\in\up{}$ then the reverse is $g-yf$, so
  $g\ulace -f$.
 Next, if $f\ulace g$ and $f\ulace h$ then $f+yg$ and
 $f+yh$ are in $\up{}$, so their product
\[f^2 + yf(g+h) + y^2 gh\in\up{}\]
Thus $f^2\ulace f(g+h)$ which implies that $f\ulace
g+h$.

(4) is similar to (3). For (5), apply (4) to $f\ulace g$ and
$-h\ulace g$.
\end{proof}

\begin{aside}
The last property is especially useful. For instance, we have
recurrences that are just like the recurrences for orthogonal
polynomials in one variable.
\end{aside}

\begin{fact}\label{recur}
  Suppose $f_0=1$, $f_1 = \sum_1^d a_ix_i+b$ where the $a_i$ are
  positive. If all constants $a_{nk}$ and $c_k$ are positive and
\[
p_{n+1} = \bigl(\sum_k a_{nk}x_k + b_n\bigr)f_n - c_n f_{n-1} 
\]
then
\[ \cdots \ulace p_n \ulace p_{n-1} \ulace
\cdots p_1 \ulace p_1
\]
\end{fact}
\begin{proof}
  We prove by induction that $p_n\ulace p_{n-1}$. This follows
  from the interlacings
\[
\bigl(\sum a_{nk}x_k + b_n\bigr)f_n \ulace f_n \ulace
c_n f_{n-1} 
\]
and Fact~\ref{basic int}.
\end{proof}

\begin{definition}
  If $f(\xx)$ is a polynomial then $f^H(\xx)$ is the sum  of all terms
  with highest total degree. 
\end{definition}

\begin{lemma}Suppose that $f(\xx)\in\up{d}$ is an upper polynomial of
  degree $n$.
  \begin{enumerate}
  \item $f^H(\xx)$ is homogeneous, and an upper polynomial.
  \item $f^H$ is the limit of homogeneous upper polynomials such that
    all monomials of degree $n$ have non-zero coefficient.
  \item All the coefficients of $f^H$ have the same argument.
  \end{enumerate}  
\end{lemma}
\begin{proof}
  The first part follows from Hurwitz's theorem and the fact that
  $f^H(\xx)$ equals $\lim_{\epsilon\rightarrow0}
  \epsilon^nf(\xx/\epsilon)$. For the next part, define
\[
f_\epsilon = f\bigl(\sum_{j=1}^d \epsilon_{1j}x_i,\dots\sum_{j=1}^d \epsilon_{dj}x_i\bigr)
\]
where all $\epsilon_{ij}$ are positive. By Fact~\ref{elem}
$f_\epsilon$ is an upper polynomial, and it converges to $f$ as we
let $\epsilon_{ii}\rightarrow 1$ and $\epsilon_{ij}\rightarrow 0$, for
$i\ne j$. For index sets $\sdiffi,\sdiffj,\sdiffk$ of degree $n$ each
non-zero monomial $\xx^\sdiffi$ in $f_\epsilon$ contributes a non-zero
coefficient to $\xx^\sdiffj$, which is different from the contribution
of $\xx^\sdiffk$, for $\sdiffi\ne\sdiffk$. Thus, the coefficient of
$\xx^\sdiffj$ is a non-zero polynomial in the $\epsilon_{ij}$'s, and
hence is non-zero for $\epsilon_{ii}$ close to $1$, and
$\epsilon_{ij}$ close to zero ($i\ne j$). 

By the second part we may assume that all coefficients of monomials of
degree $n$ are non-zero. For any index set with $|\sdiffi|=n-1$ the
polynomial $\partial_{\xx^\sdiffi}f(\xx)$ is linear, so all the
coefficients have the same argument. It follows that if
$\sdiffi,\sdiffj$ satisfy $|\sdiffi-\sdiffj|=1$ then the coefficients
of $\xx^\sdiffi$ and $\xx^\sdiffj$ have the same argument. Since all
the monomials of degree $n$ have non-zero coefficient, it follows that
all these coefficients have the same argument.
\end{proof}

\begin{aside}
We can determine if a polynomial is an upper polynomial,  if two
upper polynomials interlace, or if two real upper polynomials are constant
multiples of one another by reduction to properties of polynomials
of one variable.
\end{aside}

\begin{fact}\label{reduction} \ 
  \begin{enumerate}
  \item $f(\xx)\in\up{d}$ iff $f(\aaa+x \bbb)\in\up{1}$ for all vectors
    $\aaa$, and all vectors $\bbb$ with all positive coordinates.
  \item $f(\xx)\ulace g(x)$ iff $f(\aaa+x\bbb)\ulace
    g(\aaa+x\bbb)$ for all vectors $\aaa$, and all vectors $\bbb$ with
    all positive coordinates.
  \item Suppose that $f,g\in\up{}$ have all real coefficients. If for
    all $\aaa$ and $\bbb>0$ there is a constant $c_{\aaa,\bbb}$ so
    that $f(\aaa+x\bbb) = c_{\aaa,\bbb}g(\aaa+x\bbb)$ then $f$ and $g$
    are constant multiples of one another.
  \end{enumerate}
\end{fact}
\begin{proof}
  We begin with (1). The first direction follows from
  Fact~\ref{elem}. Conversely, suppose that $\sigma_1,\dots,\sigma_d$
  are in the upper half plane. If we choose $\sigma$ to have smaller
  positive imaginary part than any of $\sigma_1,\dots,\sigma_d$ then
  we can find $a_i$ and positive $b_i$ so that $\sigma_i = a_i + b_i
  \sigma$. Thus 
\[ f(\sigma_1,\dots,\sigma_d) = f(\aaa+ \sigma \bbb)\ne0 \]
since $f(\aaa+ y \bbb)\in\up{1}$.

For the second one, one direction is trivial. Conversely, assume that
$f(\aaa+x\bbb)\ulace g(\aaa+x\bbb)$ for all $\aaa$ and $\bbb$ as
before. By definition this means that $f(\aaa+x\bbb)+y
g(\aaa+x\bbb)\in\up{2}$.  By Fact~\ref{elem} we can substitute $a+ b
x$ for $y$, and thus by the first part $f(\xx) + y g(\xx)\in\up{}$.

For (3), first observe that the constant term of $f(\aaa+x\bbb)$ is
$f(\aaa)$.  The leading coefficient of $f(\aaa+x\bbb)$ is
$f^H(\bbb)$. It follows that
\[  f(\aaa) = c_{\aaa,\bbb} g(\aaa)\quad\text{and}\quad
f^H(\bbb) = c_{\aaa,\bbb}g(\bbb)
\]
Since all coefficients of $g^H$ have the same argument we can choose
vectors $\aaa',\bbb'>0$ so that $g(\aaa')\ne0$ and
$g^H(\bbb')\ne0$. It follows that
\[
c_{\aaa,\bbb} = c_{\aaa,\bbb'} = c_{\aaa',\bbb'} 
\]
so that $c_{\aaa,\bbb}$ is constant, which yields the conclusion.
\end{proof}

\begin{aside}
 Interlacing is essentially reflexive.
  \end{aside}

\begin{fact}\label{reflexive}
  Suppose $f,g\in\up{}$ 
  \begin{enumerate}
  \item $f^2+g^2\in\up{}$ iff $f$ and $g$ are constant multiples of
    one another.
  \item If $f\ulace g$ and $g\ulace f$ then  $f$ and $g$ are
    constant multiples of one another.    
  \end{enumerate}
\end{fact}
\begin{proof}
  Choose $\aaa,\bbb>0$. If $f(\aaa+x\bbb)^2 + g(\aaa+x\bbb)^2$ has all
  real roots then clearly $f(\aaa+x\bbb)$ and $g(\aaa+x\bbb)$ are constant multiples of each
  other. It follows that $f$ and $g$ are constant multiples also.

  In the second part we know $f+yg$ and $g+yf$ are  upper
  polynomials so
\[
(f+yg)(g+yf) = fg + (f^2+g^2)y + fgy^2 \in\up{}
\]
The coefficient of $y$ is a  upper polynomial, so the first part
finishes the proof.
\end{proof}

\begin{aside}
The real and complex parts of an upper polynomial are also upper
polynomials. In one variable this is the well-known 
\end{aside}
\begin{theorem}[Hermite-Biehler]
If $f(x)$ is a polynomial with no roots in the upper half plane and we
write $f(x) = g(x)+\imag h(x)$ where $g$ and $h$ have all real
coefficients then $g\ulace h$.
\end{theorem}

\begin{fact}\label{hb}
  Suppose $f(\xx)$ is a polynomial and we write $f(\xx) = g(\xx) + \imag h(\xx)$
  where $g$ and $h$ have all real coefficients. Then $f\in\up{}$ iff
  $g\ulace h$. 
\end{fact}
\begin{proof}
  If $g\ulace h$ then $g + y h\in\up{}$, and if we substitute
  $y=\imag$ we find that $f\in\up{}$. Conversely, choose vectors
  $\aaa$ and positive $\bbb$. We know that $f(\aaa+x \bbb) =
  g(\aaa+x\bbb) + \imag h(\aaa+x\bbb)\in\up{1}$. By the
  Hermite-Biehler theorem for one variable we conclude that 
  $g(\aaa+x\bbb)\ulace h(\aaa+x\bbb)$, and therefore
  $g(\xx)\ulace h(\xx)$.
\end{proof}

\section{Real coefficients}
\label{sec:real-coefficients}

\begin{definition}
  $\rup{d}$ consists of all polynomials in $\up{d}$ with all real
  coefficients. $\rup{}$ is all polynomials in $\up{}$ with all real
  coefficients. We call such polynomials \emph{real upper
    polynomials}.
\end{definition}

\begin{aside}
The real upper polynomials in one variable are those polynomials with
all real coefficients and all real roots. 
\end{aside}

\begin{aside}
  Interlacing is equivalent to closure under linear combinations. In
  one variable this is called Obreschkoff's theorem\cite{rahman}.
\end{aside}

\begin{fact}\label{linearity}
  Suppose $f,g\in\rup{}$. The following are equivalent
  \begin{enumerate}
  \item $\alpha f+\beta g\in\rup{}\cup \{0\}$ for $\alpha,\beta\in\reals$
  \item Either $f\ulace g$ or $g\ulace f$.
  \end{enumerate}
\end{fact}
\begin{proof}
  (2) implies (1) follows from Fact~\ref{basic int}. Conversely,
  choose vectors $\aaa,\bbb>0$ and let $T=\aaa+x\bbb$. Since $\alpha
  f(\xx)+\beta g(\xx)\in\rup{}\cup\{0\}$ we know $\alpha
  f(T)+\beta g(T)\in\rup{1}\cup\{0\}$. By Obreschkoff's theorem we know
  that either $f(T)\ulace g(T)$ or $g(T)\ulace f(T)$.  If only one
  of these possibilities occurs then $f(\xx)$ and $g(\xx)$ interlace.

  If both of these possibilities occur then by continuity we can find
  a $T$ for which we have $f(T)\ulace g(T)$ and $g(T)\ulace
  f(T)$. It follows that $f(T)=\gamma g(T)$ for some real $\lambda$,
  and hence $(f-\lambda g)(T)=0$. Since $f-\lambda g\in\rup{}\cup\{0\}$
  it follows from Fact~\ref{elem} that $f-\lambda g=0$.
\end{proof}

\begin{aside}
We  have a general construction of real upper polynomials.
The following lemma is easily proved.
\end{aside}

\begin{lemma}
    If $A,B$ are symmetric and either one is positive definite then
  $|A+\imag B|$ is not zero.
\end{lemma}

\begin{fact}\label{determinants}
If $D_1,\dots,D_d$ are positive definite $n$ by $n$ matrices, $E$ is positive semi-definite, and $S$
is symmetric then
\[ \bigl| S + \sum x_i D_i\bigr| \in\rup{d}\qquad
\bigl| S + \imag E + \sum x_i D_i\bigr| \in\up{d}  \]
\end{fact}
\begin{proof}
    If $\sigma_1,\dots,\sigma_d$ are in the upper half plane and 
  $\sigma_k = \alpha_k+\imag \beta_k$ then
\begin{gather*} 
\bigl|  S + \sum_1^d x_k\,D_k\bigr|(\sigma_1,\dots,\sigma_d) 
=  \bigl| S + \sum \sigma_k\,D_k\bigr| \\
=  \bigl| \bigl(S + \sum \alpha_k\,D_k\bigr) + \imag\bigl(
\sum \beta_k\,D_k\bigr)\bigr| 
\end{gather*}
and this is non-zero by the lemma since $\beta_k>0$ and so $\sum
\beta_k D_k$ is positive definite.  The second part is similar, and
uses the fact that $E+\sum \beta_k D_k$ is positive definite.
\end{proof}

\begin{remark}
  The Lax conjecture (now solved \cite{lax}) gives a converse for
  $n=2$. It says that if $f\in\rup{2}$ then we can find symmetric $A$
  and positive semi-definite $D_i$ so that $|A + x D_1 + y D_2| =
  f(x,y)$.
\end{remark}

\begin{aside}
Next is a generalization of Fact~\ref{determinants} that involves
sums of determinants \cites{bbs-johnson,johnson}. If $S\subset\{1,\dots,n\}$
and $A$ is an $n$ by $n$ matrix then $A[S]$ is the submatrix of $A$
whose rows and columns are indexed by $S$.
\end{aside} 

\begin{fact}\label{johnson}
  Suppose $L_k = \sum_1^d x_iD_{ik}+A_k$ where the $D_{ik}$ are $n$ by
  $n$ positive definite matrices and the $A_k$ are symmetric. The following
  is a real upper polynomial
\begin{equation}\label{eqn:johnson}
\sum_{S_1\sqcup \cdots \sqcup S_m = \{1,\dots,n\}} \bigl|
L_1[S_1]\bigr| \cdots
\bigl| L_m[S_m]\bigr|
\end{equation}
\end{fact}
\begin{proof}
If $W=diag(w_1,\dots,w_n) $then Fact~\ref{determinants} shows that
$|W+L_k|$ is an upper polynomial. Thus the reverse with respect to
$w_1,\dots,w_n$ is $|I - WL_k|$ and is also an upper polynomial. The
product
\[
\prod_{k=1}^n | I - WL_i| = \prod_{i=1}^n \sum_{S\subset\{1,\dots,n\}}
(-w)^{|S|} \, \bigl| L_k[S]\bigr|
\]
is an upper polynomial, and the coefficient of $w_1\cdots w_n$ is \eqref{eqn:johnson}.
\end{proof}

\begin{aside}
If the coefficients are  real we can reverse all the variables --
without a minus sign.
\end{aside}

\begin{fact}\label{full reverse}
If $f(\xx)\in\rup{d}$ and if $x_i$ has degree $e_i$ then $x_1^{e_1}\cdots x_d^{e_d}
  f(1/x_1,\dots,1/x_d)\in\rup{d}$. 
\end{fact}
\begin{proof}
  If $\sigma_i$ is in the upper half plane then
  $1/\overline{\sigma_i}$ is also. Thus 
\[ f(1/\sigma_1,\dots,1/\sigma_d) =
\overline{f(1/\overline{\sigma_1},\dots, 1/\overline{\sigma_d})} \ne 0.
\]
\end{proof}

\begin{fact}\label{multiaff}
  If $a,b,c,d$ are real then $f(x,y)=a+bx + cy + d xy\in\rup{2}$ iff 
$\smalltwodet{b}{a}{d}{c}\ge0$.
\end{fact}
\begin{proof}
 If the determinant is zero then there is a $\lambda$ so that
 $b\lambda=d$ and $a\lambda=c$, and so
$f(x,y)= (a+bx)(1+\lambda y)$
is in $\up{2}$. If the determinant is not zero then solving $f=0$
yields
\[
y = -\frac{a+bx}{c+dx}
\]
If $M$ is the \Mobius\ transformation with matrix
$\smalltwodet{b}{a}{d}{c}$, $f\in\up{2}$ and $x$ is in the upper
half plane then $-Mx$ is in the complement. Thus $M$ maps the upper
half plane to itself. This happens exactly when the determinant is
positive. 
\end{proof}

\begin{aside}
This result can be generalized \cite{branden-hpp}.
\end{aside}

\begin{fact}
  Suppose that $f(\xx)$ is a polynomial where every variable has
  degree $1$. Then $f(\xx)$ is in $\rup{d}$ iff
\[
\frac{\partial f}{\partial x_i}\cdot\frac{\partial f}{\partial x_j}-
f\cdot
\frac{\partial f}{\partial x_i\partial x_j}\ge0
\]
for all $\xx\ge0$ and $1\le i,j\le d$.
\end{fact}

\begin{aside}
The next two facts are simple consequences of Fact~\ref{multiaff}.
\end{aside}

\begin{fact}\label{f and g}
  If $f\ulace g$ and $f,g\in\rup{d}$ then 
  $\smalltwodet
    {f} {g} {\frac{\partial f}{\partial x_1}}{\frac{\partial g}{\partial x_1} }
    \le0$.
\end{fact}
\begin{proof}
  Since $f+yg\in\rup{}$ so is the Taylor series
\begin{multline*}
f(x_1+z,x_2,\dots,x_d) + y g(x_1+z,x_2,\dots,x_d) = \\
f(\xx) + z \frac{\partial f}{\partial x_1} + y g(\xx) + y z
\frac{\partial g}{\partial x_1} + \cdots
\end{multline*}
Reversing, differentiating, and reversing  shows that 
\[
f(\xx) + z \frac{\partial f}{\partial x_1} + y g(\xx) + y z
\frac{\partial g}{\partial x_1} \in\rup{}
\]
so we can apply Fact~\ref{multiaff}.
\end{proof}

\begin{aside}
The next fact can be stated for $d$ variables, but it is fundamentally
a property of two variable polynomials.
\end{aside}

\begin{fact}\label{quad-inequality}
  If $\sum a_{i,j}x^iy^j \in\rup{2}$ then 
$  \begin{vmatrix}
    a_{r+1,s} & a_{r+1,s+1} \\ a_{r,s} & a_{r,s+1}
  \end{vmatrix}\le 0
$
for $r,s\ge0$.
\end{fact}
\begin{proof}
  We differentiate, reverse, and differentiate first with respect to
  $x$ and then to $y$ so that the $x$ and $y$ degrees equal $1$. The
  result is
\[
r!s! a_{r,s} + (r+1)!s! a_{r+1,s}x + 
r!(s+1)! a_{r,s+1}y + (r+1)!(s+1)! a_{r+1,s+1} xy
\]
By Fact~\ref{multiaff} 
\begin{multline*}
0\ge
\begin{vmatrix}
r!s! a_{r,s} &  (r+1)!s! a_{r+1,s} \\
r!(s+1)! a_{r,s+1} & (r+1)!(s+1)! a_{r+1,s+1}  
\end{vmatrix} 
\\
= 
r!s!(r+1)!(s+1)!
\begin{vmatrix}
 a_{r,s} &   a_{r+1,s} \\
 a_{r,s+1} &  a_{r+1,s+1}  
\end{vmatrix} 
\end{multline*}
\end{proof}

\begin{remark}\label{rem:interlacing}
  Interlacing in $\rup{1}$ is closely related to the usual definition
  of interlacing in terms of the location of roots. Suppose
  $f,g\in\rup{1}$, and $f$ has positive leading coefficient. We assume
  the roots of $fg$ are all distinct and that the roots of $f$ and $g$
  alternate, with the largest root belonging to $f$. There are
  various cases depending on the sign of the leading coefficient of $g$, and the
  degrees.\\

  \begin{tabular}{cccl}
    \toprule
    Degree of $f$ & Degree of $g$ & Leading coefficient of $g$ &
    Interlacing \\\midrule
    n & n-1 & + & $f\ulace g$ \\
    n & n-1 & - & $f\ulace -g$ \\
    n & n & + & $f\ulace g$ \\
    n & n & - & $f\ulace -g$ \\
    \bottomrule
  \end{tabular}
\ \\

  In one direction this is proved by explicit constructions that
  $f+yg$ (and its variants) are in $\rup{2}$. In the other direction
  we use  Obreschkoff's theorem to conclude that they interlace in
  the traditional sense. The reflexivity of interlacing implies only
  the answer in the table is possible.
\end{remark}

\section{Analytic closure}
\label{sec:analytic-closure}

\begin{definition}
  $\upbar{d}$ is the uniform closure on compact subsets of $\up{d}$.
\end{definition}

\begin{fact}\label{e-xy}
  $e^{-\xx\cdot\yy} \in\upbar{2d}$ and $e^{-\xx\cdot\xx}\in\upbar{d}$.
\end{fact}
\begin{proof}
We know $1-x_1y_1/n\in\up{2}$, so $\lim_{n\rightarrow\infty}
(1-x_1y_1/n)^n = e^{-x_1y_1}\in\upbar{2}$. Closure under
multiplication implies 
\[ e^{-\xx\cdot\yy} = e^{-x_1y_1}\cdots e^{-x_dy_d}\in\upbar{2d}. \]  
Setting $\xx=\yy$ establishes the second part.  
\end{proof}

\begin{fact}\label{e-dx dy}
  If $f(\xx,\yy)\in\up{2d}$ then $e^{-\partial_\xx\cdot\partial_\yy}f(\xx,\yy)\in\up{2d}$.
\end{fact}
\begin{proof}
  If $f\in\up{}$ then the interlacing of derivatives implies
\[ f \ulace \partial_{x_i} \ulace
\frac{1}{n}\partial_{y_i}\bigl(\partial_{x_i} f\bigr)
\]
and from Fact~\ref{basic int}
\[f - \frac{1}{n} \partial_{y_i}\partial_{x_i} f =
\biggl(1-\frac{\partial_{x_i}\,\partial_{y_i}}{n}\biggr)f\in\up{}
\]
Iterating this shows that the map $f\mapsto
\biggl(1-\frac{\partial_{x_i}\,\partial_{y_i}}{n}\biggr)^nf$ sends
$\up{}$ to itself. Taking limits shows that
$f\mapsto e^{-\partial_{x_i}\,\partial_{y_i}}f$ also maps $\up{}$ to
itself. Composing these transformations for $i=1,\dots,d$ yields
\[
e^{-\partial_\xx\cdot\partial_\yy}f =
e^{-\partial_{x_1}\,\partial_{y_1}} \cdots
e^{-\partial_{x_d}\,\partial_{y_d}}f\in\up{}
\] 
\end{proof}

\begin{fact}\label{f(d)}
  If $f(\xx)\in\up{d}$ and $g(\xx)\in\up{d}$ then
  $f(-\partial_\xx)g(\xx)\in\up{d}\cup\{0\}$. This also holds for $f\in\upbar{d}$.
\end{fact}
\begin{proof}
      Since $e^{-\partial_\xx\partial_\yy}$ maps $\up{2d}$ to
    itself the lemma follows from the identity
\[
e^{-\partial_\xx\partial_\yy}\,g(\xx)\,f(\yy)\biggl|_{\yy=0} \,=\,
f(-\partial_\xx)g(\xx).
\]
By linearity we only need to check it for monomials since
$e^{-\partial_\xx\partial_\yy}g(\xx)f(\yy)\in\up{}$, and evaluation at
$\yy=0$ is in $\up{}\cup\{0\}$.
\[
e^{-\partial_\xx\partial_\yy} \xx^\sdiffi \yy^\sdiffj\biggl|_{\yy=0} \, =\,
\frac{(-\partial_\xx)^\sdiffj}{\diffj!} \xx^\sdiffi \,\cdot\, (\partial_\yy)^\sdiffj
\yy^\sdiffj\biggl|_{\yy=0}\, =\,
(-\partial_\xx)^\sdiffj \xx^\sdiffi
\]
 All polynomials involved have bounded degree so the passage to the
limit presents no problems.
\end{proof}

\begin{fact}\label{preserver}
  Suppose $T$ is a non-trivial linear transformation defined on polynomials in $d$
  variables.  $T(e^{-\xx\cdot\yy})\in\upbar{2d}$ if and only if $T$ maps
  $\up{d}\cup\{0\}$ to itself.
\end{fact}
\begin{proof}
  Since $f(-\partial_\yy)$ maps $\upbar{}\cup\{0\}$ to itself one direction
  follows from the identity
\[
f(-\partial_\yy)\, T(e^{-\xx\cdot\yy})\bigl|_{y=0} = T(f).
\]
We only need to verify this on  monomials:
\[
(-\partial_{\yy^\sdiffj})T(e^{-\xx\cdot\yy})\bigl|_{y=0} =
(-\partial_{\yy^\sdiffj})\sum_\sdiffi T(\xx^{\sdiffi})
\frac{(-\yy)^{\sdiffj}}{\diffj!} \bigl|_{y=0} =
T(\xx^\sdiffi)
\]
The converse can be found in \cite{bb2}, where $T$ non-trivial means
that $T(\up{})$ has dimension at least $3$.
\end{proof}

\begin{example}
  The expression $T(x^{-\xx\cdot\yy})$ is known as the \emph{generating
  function} of $T$. 
  If $H_i(x)$ is the $i$'th Hermite polynomial then the linear
  transformation
\[ T:x_1^{i_1}\cdots x_d^{i_d} \mapsto H_{i_1}(x_1)\cdots
H_{i_d}(x_d) \]
has generating function
\[
\prod_{i=1}^d e^{-2x_iy_i - y_i^2} = e^{-2\xx\cdot\yy-\yy\cdot\yy} \]
Since this is in $\upbar{2d}$ it follows that $T$ maps $\up{d}$ to itself.
\end{example}

\begin{fact}\label{f(x,D)}
  Suppose that $f(\xx,\yy)$ is a polynomial, and define $T(g) =
  f(\xx,-\partial_\yy)g$. The following are equivalent.
  \begin{enumerate}
  \item $T:\up{d}\longrightarrow \up{d}\cup{0}$.
  \item $f(\xx,\yy)\in\up{2d}$.
  \end{enumerate}
\end{fact}
\begin{proof}
  The generating function of $T$ is
\[ f(\xx,-\partial_\yy)e^{-\xx\cdot\uu -
  \yy\cdot\vv}\bigl|_{\uu=\vv=0} = 
f(\xx,\vv)
\]
so the result follows from Fact~\ref{preserver}.
\end{proof}

\section{Uniqueness results}
\label{sec:uniqueness-results}

\begin{aside}
We end with some interesting uniqueness results that we present
with only sketches of the proof.
\end{aside}

\begin{fact}\label{ps}\cite{bb2}
  Suppose that $T$ is a linear transformation that maps
  $\rup{d}\longrightarrow\rup{d}\cup\{0\}$. If $T(\xx^\sdiffi) =
  \aaa_\sdiffi \xx^\sdiffi$ then $T$ is a composition of
  one-dimensional transformations. That is, there are linear
  transformations $T_i\colon \rup{1}\longrightarrow\rup{1}\cup\{0\}$
  of the form $T_i(x_i^k) = a_{ik}x_i^k$ and $T=T_1\cdots T_d$.
\end{fact}
\begin{proof}
 It suffices to assume $d=2$, so write $T(x^iy^j)=a_{i,j}x^iy^j$. The
 inequality of Fact~\ref{quad-inequality} applied to
 $T(x^iy^j(1+x)(1+y))$ and $T(x^iy^j(1+x)(1-y))$ yields that
 $a_{ij}a_{i+1,j+1}=a_{i,j+1}a_{i+1,j}$. We use this identity to prove
 inductively that $a_{ij}=a_{i0}a_{0j}$. $T$ is now the product of the
 two transformations $x^i\mapsto a_{i,j0}x^i$ and $y^j\mapsto a_{0,j}y^j$.
\end{proof}

\begin{fact}\label{u-deriv}\cite{fisk}
  If $T$ is a linear transformation on $\rup{d}$ such that
  \begin{enumerate}
  \item $T$ reduces degree.
  \item $f \ulace Tf$
  \end{enumerate}
then there are $a_1,\dots,a_d$ of the same sign so that
\[ Tf = a_1\frac{\partial f}{\partial x_1} + \cdots + a_d\frac{\partial f}{\partial x_d}\]
\end{fact}
\begin{proof}
  We only discuss $d=1$. We use the fact that $\alpha(x+b)^{n-1}$ is
  the only polynomial of smaller degree interlacing
  $(x+b)^n$. Choosing $f=x^n$ shows that $T(x^n)=
  a_nx^{n-1}$. Choosing $f=(x+1)^n$ shows
\[
T(x+1)^n = \sum a_i\binom{n}{i}x^{i-1} = \alpha \sum \binom{n-1}{i}x^i
\]
and equating coefficients yields $T(x^n) = a_1 nx^{n-1}$.
\end{proof}

\begin{fact}\label{affine} \cite{fisk}
  If $T$ is a linear transformation such both $T$ and $T^{-1}$ map
  $\rup{1}$ to itself then $(Tf)(x) = a\,f(bx+c)$ where $ab\ne0$.
\end{fact}
\begin{proof}
  We may assume that $T(1)=1$ and $T(x)=x$.
  Since $T(x^n) \ulace T(x^n)'$ we have that $x^n \ulace
  T^{-1}(T(x^n)')$, so the linear transformation $S(f)=T^{-1}(T(f)')$
  satisfies the hypothesis of the previous fact. Consequently
  $T^{-1}(T(x^n)') = \alpha_n (x^n)'$ and hence $T(f')= (Tf)'$. We use
  this inductively to show that $T(x^n)=x^n$.
\end{proof}

\begin{fact}\label{even}\cite{fisk}
  If $f(x,y)\in\rup{2}$ has the property that all exponents of $x$ and
  $y$ are even then $f(x,y) = g(x)h(y)$ with $g,h\in\rup{1}$.
\end{fact}
\begin{proof}
  This is a special case of the following result:
  \begin{quote}
    If $f(x,y) = \cdots + f_{i-1}(x)y^{i-1} + 0\cdot y^i +
    f_{i+1}(x)y^{i+1}+\cdots\in\rup{2}$ where $f_{i-1}f_{i+1}\ne0$ then $f(x,y)
    = g(x)h(y)$ where $f,g\in\rup{1}$.
  \end{quote}
  To prove this we differentiate, reverse, and differentiate so that
  we only have $f_{i-1}(x) + f_{i+1}(x)y^2$. It follows that $f_{i+1}$
  is a constant multiple of $f_{i-1}$.  Continuing this argument
  concludes the proof.
\end{proof}

\section{Questions}
\label{sec:questions}

\begin{aside}
  Here are a few unsolved questions.
\end{aside}
\begin{question}
  Suppose that $f(\xx),g(\xx)\in\rup{d}$ have the property that
  $f+\alpha g\in\rup{d}$ for all positive $\alpha$. Show that there is
  an $h\in\rup{d}$ so that $h\ulace f$ and $h\ulace g$. This is easy
  if $d=1$.
\end{question}
\begin{question}
  If $T$ is a bijection on $\rup{d}$ then are there  constants
  $\aaa>0$, $\bbb$,$\alpha$ and a permutation $\sigma$ of $\{1,\dots,d\}$ so
  that \[T(f(\xx)) = \alpha \,f(a_1x_{\sigma1}+b_1,\dots,a_d x_{\sigma d}+b_d)?\]
\end{question}
\begin{question}
  
\end{question}

\end{document}